\documentclass{amsart}

\usepackage{amssymb,amsfonts, amsmath, amsthm}
\usepackage[all,arc]{xy}
\usepackage{enumerate}
\usepackage{mathrsfs}
\usepackage{mathtools}
\usepackage[mathscr]{euscript}
\usepackage{array}
\usepackage{tikz}
\usepackage{tikz-cd}
\usepackage{textcomp}
\usetikzlibrary{decorations.pathmorphing,shapes}
\usepackage[colorlinks=true, linkcolor=blue, citecolor = red]{hyperref}
\usetikzlibrary{patterns}

\usepackage{geometry}

\newcommand{\C}{\mathscr{C}}

\newcommand{\E}{\mathscr{E}}

\newcommand{\Y}{\mathscr{Y}}

\newcommand{\U}{\mathscr{U}}

\renewcommand{\Y}{\mathscr{Y}}

\newcommand{\spaces}{\mathscr{S}\text{paces}}

\newcommand{\comma}{,}

\newcommand{\set}{\mathscr{S}\text{et}}

\newcommand{\pbsq}[8]{
  \begin{tikzcd}[row sep=0.5in, column sep=0.5in]
    #1 \arrow[r, "#5"] \arrow[d, "#6"'] \arrow[dr, phantom, "\ulcorner", very near start]
    \pgfmatrixnextcell #2 \arrow[d, "#7"] \\
    #3 \arrow[r, "#8"']
    \pgfmatrixnextcell #4
  \end{tikzcd}
}

\newtheorem{theone}{Theorem}[section]
\newtheorem{lemone}[theone]{Lemma}

\theoremstyle{definition}
\newtheorem{defone}[theone]{Definition}

\theoremstyle{remark}
\newtheorem{remone}[theone]{Remark}

\makeatletter
\def\@seccntformat#1{%
  \expandafter\ifx\csname c@#1\endcsname\c@section\else
  \csname the#1\endcsname\quad
  \fi}
\makeatother

\title{Yoneda Lemma for Elementary Higher Toposes}

\author{Nima Rasekh}

\date{September 2018}

\begin{document}

\begin{abstract}

We prove the Yoneda lemma inside an elementary higher topos, generalizing the Yonda lemma for spaces.

\end{abstract}

\maketitle
\addtocontents{toc}{\protect\setcounter{tocdepth}{1}}

\tableofcontents

 \section{Introduction} \label{Introduction}
  
 In classical category theory the Yoneda lemma states that there is an embedding from each category 
 to the category of set-valued presheaves
 $$\C \to Fun(\C^{op}, \set).$$
 In the world of higher categories this has been generalized to an embedding from an $(\infty,1)$-category 
 to the $(\infty,1)$-category of space-valued presheaves
 $$\C \to Fun(\C^{op}, \spaces)$$
 One particular instance is the case when $\C$ is an $\infty$-groupoid, which is the data of a space. 
 In this case $\C$ is equivalent to $\C^{op}$ and thus we get an embedding
 $$ \C \to Fun(\C, \spaces).$$
 Thus we can always embed a space in the $(\infty,1)$-category of space valued functors.
 \par 
 An elementary higher topos is an $(\infty,1)$-category that shares many characteristics with the category of spaces.
 The goal of this paper is to show that we can use those properties to show that the Yoneda embedding for spaces 
 also holds in an elementary higher topos.

%
 

 
\section{The Yoneda Lemma for Elementary Toposes}
 Before we tackle the higher categorical setting we review the $1$-categorical case first.
 In order to make things concrete we will start with sets first.
 \par 
 Let $S$ be a set. 
 The diagonal map $\Delta: S \to S \times S$ is always a monomorphism of sets.
 We can think of this map as the reflexive relation. In other words 
 $$ (x,y) \in \Delta(S) \Leftrightarrow x = y$$
 As it is a subset it has a characteristic function $\chi_{\Delta}: S \times S \to \{ 0,1 \}$ and pullback square
 \begin{center}
  \pbsq{S}{ \{ 1 \} }{S \times S}{ \{ 0 \comma 1 \} }{}{\Delta}{1}{\chi_{\Delta}}
 \end{center}
 The map $\chi_{\Delta}$ sends a pair $(x,y)$ to $1$ if and only if $x = y$.
 This characteristic map has an adjoint 
 $$ S \to \{ 0,1 \}^S$$
 In order to characterize this map we first realize that $\{ 0,1 \}^S$ is the powerset of $S$, $PS$.
 Thus we are looking for a map from $S$ to its powerset.
 From this perspective the map can be characterized as the map that sends an element $s \in S$ to the subset $\{ s \} \subset S$.
 We will thus denote this map by 
 $$\{ \cdot \}_S : S \to PS$$
 This map is itself a monomorphism. There is a categorical way to think about this monomorphism.
 If $\C$ is a $1$-category then by the Yoneda lemma there is a embedding from $\C$ into the category of presheaves
 $$ \Y : \C \to Fun(\C^{op}, \set)$$
 If $\C$ is a set (a category with only identity maps) then we still have the Yoneda embedding,
 however every representable presheaf takes values in the subcategory $ \{ \emptyset , \{ 1 \} \} $.
 Indeed for any two objects $c,d$ in $\C$ we either have $Hom_\C(c,d) = \emptyset$ or $Hom_\C(c,d) = \{ 1 \}$.
 Moreover, also notice that if $\C$ is a set then $\C^{op} = \C$
 \par 
 Combining these two facts we get an embedding of categories 
 $$ \Y : \C \to Fun( \C, \{ \emptyset , 1 \} )$$
 However, at this stage we are dealing with sets and the set of maps from $\C$ to $\{ \emptyset , 1 \}$ is just the powerset.
 Recall that $\Y (c): \C \to \{ \emptyset , 1 \}$ is the map that takes $c$ to $1$ and everything else to $\emptyset$.
 Thus the corresponding powerset is the subset $\{ c \}$.
 \par 
 What we have shown in the previous paragraph is that the Yoneda embedding $\Y$ is equal to the 
 singleton map $\{ \cdot \}_\C$. 
 The most interesting aspect of this result is that it can be generalized to an elementary topos,
 which is a $1$-category that shares characteristics with the category of sets.
 \par 
 Here we will only give the definition of an elementary topos and refer to \cite{MM92} or \cite[Subsection 1.2]{Ra18c} for more details.
 \begin{defone}
  An elementary topos $\E$ is a $1$-category which is locally Cartesian closed and has a subobject classifier $\Omega$.
 \end{defone}
 As before, for any object $A$ in $\E$, the diagonal map $A \to A \times A$ is a monomorphism.
 The subobject classifier $\Omega$ characterizes all monomorphisms (which is the role $\{ 0,1 \}$ played in sets).
 Thus we have a pullback square in $\E$.
 \begin{center}
  \pbsq{A}{1}{A \times A}{\Omega}{}{\Delta}{t}{\chi_\Delta}
 \end{center}
 Using adjunctions we get the map, which we will also denote by $\{ \cdot \}_A$. 
 $$\{ \cdot \}_A : A \to PA $$
 The axioms of an elementary topos allow us to generalize the result above thusly.
 
 \begin{lemone} \label{Lemma Yoneda for ET}
  \cite[Lemma IV.1.1]{MM92}
  Let $\E$ be an elementary topos and $A$ an object. Then the map 
  $$\{ \cdot \}_A : A \to PA $$
  is a monomorphism in $\E$.
 \end{lemone}

 The goal for the coming sections is two-fold. First we review the Yoneda lemma for spaces 
 similar to the case for sets. Then we prove that it generalizes to elementary higher toposes 
 the same way the case for sets generalized to elementary toposes.
 
\section{The Yoneda Lemma for Spaces}
 
 In this section we want to see how we can witness the Yoneda lemma in the context of spaces. 
 In this case the diagonal map is replaced by the {\it path fibration}:
 $$p : X^{\Delta[1]} \xrightarrow{ \ \ (s,t)  \ \ } X \times X$$
 In order to understand it we have to gain a better understanding of spaces first.
 \begin{remone} \label{Rem U in Sp}
  Let $\spaces$ be the category of small spaces. Then we denote the maximal subgroupoid as $\spaces^{core}$ and notice that it
 is itself a space. We define $\spaces_*$ as the category of pointed small spaces and recall that there is a forgetful functor 
 $U: \spaces_* \to \spaces$. This induces a map of (large) spaces $U^{core}: (\spaces_*)^{core} \to \spaces$.
 One beautiful result about spaces is that this the map $U^{core}$ classifies other maps.
 Concretely for every space $X$ we have an equivalence of spaces with small fiber over $X$ and maps $X$ into $\spaces$
 $$(\spaces^{small}_{/X})^{core} \simeq map(X,\spaces^{core})$$
 where the equivalence is given by pulling back along $\U^{core}$. 
 In order to check that a map over $p: Y \to X$ is characterized by a certain morphism $f:X \to \spaces^{core}$ it suffices to show 
 that the fiber of $p$ over each point is equivalent to the value of $f$ at that point.
 \end{remone}
 We can define a map 
 $$Path(-,-): X \times X \to \spaces^{core}$$
 which sends each pair of points $(x,y)$ to the space of paths that start at $x$ and end with $y$, $Path_X(x,y)$.
 This maps gives us a pullback square:
 \begin{center}
  \pbsq{X^{\Delta[1]}}{\spaces_*^{core}}{X \times X}{\spaces^{core}}{}{(s \comma t)}{U^{core}}{Path_X(- \comma -)}
 \end{center}
 Indeed this follows from the fact that for every two points $x,y$ in $X$ we have a pullback square
 \begin{center}
  \pbsq{Path_X(x,y)}{X^{\Delta[1]}}{*}{X \times X}{}{}{(s \comma t)}{(x \comma y)}
 \end{center}
 In light of the remark above what we are saying is that path fibration is classified by the map $Path(-,-)$.
 \par 
 The map $Path_X(-,-): X \times X \to \spaces^{core}$ has an adjoint 
 $$Path_X(,-):X \to (\spaces^{core})^X$$
 We want to prove that this map is an embedding of spaces using the tools of algebraic topology. 
 All we have to do is to show that for any two objects $x,y$ the map of spaces 
 $$ Path_X(x,y) \to Path_{(\spaces^{core})^X}(Path_X(,x), Path_X(,y))$$
 is an equivalence (Notice the similarity to the Yoneda lemma that we will later come back to).
 \par 
 In order to do that we first use the remark above to characterize the map $Path_X(,x) : X \to \spaces^{core}$
 as a map over $X$. 
 We can see that we have the following pullback diagram 
 \begin{center}
  \pbsq{Path_X(x)}{(\spaces_*)^{core}}{X}{\spaces^{core}}{}{s_x}{U^{core}}{}
 \end{center}
 where $Path_X(x)$ is the space of paths in $X$ that end at $x$.
 Indeed the value at each point $y$ and the fiber over $y$ are both the space $Path_X(y,x)$ 
 (here we again used Remark \ref{Rem U in Sp}).
 \par 
 Thus $Path_X(,x)$ and $Path(,y)$ are equivalent to the maps of spaces from $Path_X(x)$ to $Path_X(y)$ over $X$.
 That is a commutative diagram of the following form.
 \begin{center}
  \begin{tikzcd}[row sep=0.5in, column sep=0.5in]
   Path_X(x) \arrow[dr, "s_x"'] \arrow[rr, dashed] & & Path_X(y) \arrow[dl, "s_y"] \\
   & X & 
  \end{tikzcd}
 \end{center}

 So, our original question is equivalent to showing that 
 $$ Path_X(x,y) \to Map_X(Path_X(x), Path_X(y))$$
 is an equivalence of spaces. 
 \par 
 Here we can use the power of algebraic topology. The space $Path_X(x)$ is contractible with contraction point 
 the identity path at $x$ and so any such commutative triangle above is equivalent to a commutative triangle of the form
 \begin{center}
  \begin{tikzcd}[row sep=0.5in, column sep=0.5in]
   * \arrow[dr, "x"'] \arrow[rr, dashed] & & Path_X(y) \arrow[dl, "s_y"] \\
   & X & 
  \end{tikzcd}
 \end{center}
 But the space of such diagrams is just the fiber of $s_y: Path_X(y) \to X$ over $x$, which is by definition $Path_X(x,y)$. 
 Hence we have shown that the desired map is an embedding.
 \par 
 This same result can also be witnessed form the perspective of the Yoneda lemma in higher category theory.
 If $\C$ is a higher category ( $(\infty,1)$-category), then by the Yoneda lemma we have a diagram 
 $$\Y: \C \to Fun(\C^{op}, \spaces)$$
 If $\C$ is a groupoid then it is just a space and we can simplify the diagram to 
 $$\Y: \C \to Fun(\C, \spaces)$$
 However, if $\C$ is a space then any functor $\C \to \spaces$ will factor through the maximal subgroupoid 
 $\C \to \spaces^{core} \hookrightarrow \spaces$.
 Thus we can simplify the map above to 
 $$\Y: \C \to Fun(\C, \spaces^{core}) = (\spaces^{core})^\C$$
 By the Yoneda lemma this map is also an embedding. The image of a point $x$ in $\C$ is the representable functor 
 $\Y_x: \C \to \spaces^{core}$ that takes a point $y$ to the space $map_{\C}(y,x)$. 
 However, $\C$ is a just a space and so $map_{\C}(y,x)$ is by definition $Path_\C(y,x)$. 
 So, the Yoneda embedding then corresponds to the result we stated in the previous paragraph.
 
 \begin{remone}
  It is instructive to see how the case for spaces differs from the case of sets.
  Let $X$ be a $0$-type, by which which we mean a space which has no non-trivial homotopy groups above level $0$.
  Alternatively we can characterize it as a disjoint union of contractible spaces.
  In this case the path space $Path(x,y)$ for any two points is either empty (if they are not equivalent) or contractible 
  (if they are equivalent). Then the path fibration is precisely the diagonal map we used in the 
  preious section, which implies that the case for sets is a special case of spaces.
 \end{remone}

 \section{The Yoneda Lemma for Elementary Higher Toposes}
 We now want to generalize the Yoneda lemma for spaces to the setting of an elementary higher topos.
 First we have following definition.
 
 \begin{defone} \cite[Definition 3.5]{Ra18c} \cite[Theorem 3.16]{Ra18c}
  An elementary higher topos $\E$ is an $(\infty,1)$-category that has finite limits and colimits, a subobject classifier 
  and every map $f:X \to Y$ is classified by a universe $\U$, in the sense that there is a pullback 
  \begin{center}
   \pbsq{X}{Y}{\U_*}{\U}{}{}{}{}
  \end{center}
 \end{defone}

 We want to repeat the steps we took in the previous section. The main problem is that $X^{\Delta[1]} \to X \times X$ does not exists 
 in an arbitrary topos as $\Delta[1]$ does not exists. However, notice that the map $X^{\Delta[1]} \to X \times X$ is 
 equivalent to the diagonal map $X \to X \times X$. As all constructions in a higher category are homotopy invariant
 we can thus just work with the diagonal map.
 
 \begin{remone}
  If the statement above is true then why did we choose the path fibration in the previous section?
  The path fibration map is a Kan fibration and so is normally used in algebraic topology 
  as it allows us to use the classical definition of a pullback.
  In an arbitrary elementary higher topos, however, there is no notion of fibration and so we have to work with the diagonal map.
 \end{remone}
 
 Having decided that we want to use the diagonal map, we now have to characterize it. 
 Here is where we use the fact that elementary higher toposes have universes. By assumption there is a universe $\U$ such that 
 there is following pullback diagram.
 \begin{center}
  \pbsq{X}{\U_*}{X \times X}{\U}{}{\Delta}{}{}
 \end{center}
 That is exactly the role played by $\spaces^{core}$ in the case of spaces.
 The map 
 $$X \times X \to \U$$
 has an adjoint (as $\E$ is locally Cartesian closed)
 $$X \to \U^X.$$
 We are finally in a position to state and prove the Yoneda lemma.
 
 \begin{theone} \label{The Yoneda Lemma}
  Let $X$ be an object in $\E$ and let $\U$ be a universe that classifies the diagonal map $\Delta: X \to X \times X$
  in the sense that we have following pullback square
  \begin{center}
   \pbsq{X}{\U_*}{X \times X}{\U}{}{\Delta}{}{}
  \end{center}
  Then the adjoint map 
  $$\Y_X : X \to \U^X$$
  is an embedding in $\E$.
 \end{theone}
 
 \begin{proof}
  In order to prove $\Y_X$ is an embedding we have to show that for every object $Z$ the induced map
  $$ map_\E(Z, X) \to map_\E(Z, \U^X ) $$
  is an embedding of spaces. By adjointness this is equivalent to the map 
  $$ map_\E(Z,X) \to map_\E(Z \times X, \U)$$
  being an equivalence.
  \par 
  Let $S$ be the class of maps classified by the universe $\U$. Then we have an equivalence of spaces 
  $$map(Z \times X, \U) \xrightarrow{ \ \ \simeq \ \ } (\E_{/Z \times X}^S)^{core}.$$
  Composing with the original map means that it suffices to prove that the map 
  $$ \varphi: map_\E(Z,X) \to (\E_{/Z \times X}^S)^{core}$$
  is an embedding. Notice for $f:Z \to X$ the image of $\varphi(f)$ is the map $(id_Z,f): Z \to Z \times X$.
  We will do this in two steps. First we determine the points in $(\E_{/Z \times X}^S)^{core}$
  for which the fiber of the map is non-empty. Then we show that $map_\E(Z,X)$ is equivalent to the 
  image we characterized in the previous step.
  \par 
   For the first step we will prove that a map $p:E \to Z \times X$ has a non-trivial fiber if and only if 
   there exists a map $g: Z \to X$ and equivalence over $Z \times X$
     \begin{center}
   \begin{tikzcd}[row sep=0.5in, column sep=0.5in]
    E \arrow[rr, " \simeq"] \arrow[dr, "p"'] & & Z \arrow[dl, "( id_Z \comma g )"] \\
    & Z \times X
   \end{tikzcd}
  \end{center}
  Let $p: E \to Z \times X$ be a point in $(\E_{/Z \times X}^S)^{core}$.
  The fiber over this point is the subspace of $map_\E(Z,X)$ generated by maps $g: Z \to X$ such that 
  the following is a pullback square in $\E$.
  \begin{center}
   \begin{tikzcd}[row sep=0.5in, column sep=0.5in]
    E \arrow[d, "p"] \arrow[r] \arrow[dr, phantom, "\ulcorner", very near start] & X \arrow[d, "\Delta"] \\
    Z \times X \arrow[r, "g \times id_X"] & X \times X 
   \end{tikzcd}
  \end{center}
  We can extend the diagram.
  \begin{center}
   \begin{tikzcd}[row sep=0.5in, column sep=0.5in]
    E \arrow[d, "p"'] \arrow[r] \arrow[dr, phantom, "\ulcorner", very near start] & X \arrow[d, "\Delta"] \\
    Z \times X \arrow[r, "g \times id_X"] \arrow[d, "\pi_1"'] \arrow[dr, phantom, "\ulcorner", very near start] 
    & X \times X \arrow[d, "\pi_1"] \\
    Z \arrow[r, "g"] & X
   \end{tikzcd}
  \end{center}
  As the bottom square is always a pullback square, the top one is a pullback square 
  if and only if the rectangle is a pullback square. However, $\pi_1 \circ \Delta = id_X : X \to X$ and so 
  the pullback $\pi_1 \circ g : E \to Z$ is an equivalence. This prove that one side of our condition.
  \par 
  For the other side let $g: Z \to X$ be a map. We will show the map $(id_Z,g): Z \to Z \times X$ has 
  a non-trivial fiber. First of all we can construct the pullback square
  \begin{center}
   \begin{tikzcd}[row sep=0.5in, column sep=0.5in]
    Z \arrow[d, "(id_Z \comma g)"'] \arrow[r, "g"] \arrow[dr, phantom, "\ulcorner", very near start] & X \arrow[d, "\Delta"] \\
    Z \times X \arrow[r, "g \times id_X"] & X \times X 
   \end{tikzcd}
  \end{center}
  Moreover, the map $(id_Z , g): E \to Z \times X$  is in $S$
  as it is the pullback of $\Delta_X: X \to X \times X$ which is in $S$ by assumption and $S$ is closed under pullback.
  We have just shown that $p: E \to Z \times X$ has a non-trivial fiber if and only if there exists a map $g: Z \to X$ and an equivalence 
  $ E \to Z$ over $Z \times X$.
  \par 
  Now we need to show that $map(Z,X)$ is equivalent to the image.
  The map 
  $$map(Z \times X, \U) \xrightarrow{ \ \ \simeq \ \ } (\E_{/Z \times X}^S)^{core}$$
  has a Kan fibration over the space $(\E^S)^{core}$. We can take the fiber over the point $Z: * \to (\E^S)^{core}$. 
  This gives us an equivalence
  $$map(Z \times X, \U)_{Z} \xrightarrow{ \ \ \simeq \ \ } map^S_\E(Z, Z \times X).$$
  Here $map(Z \times X, \U)_{Z}$ is the subspace of $map(Z \times X, \U)$ generated by point $g:Z \times X \to \U$
  such that $g^*(\U^*) \simeq Z$. Also, the space $map^S_\E(Z, Z \times X)$ is the subspace of $map_\E(Z, Z \times X)$
  generated by maps $f: Z \to Z \times X$ such that $f \in S$.
  In the previous paragraph we proved that the map $ \varphi: map_\E(Z,X) \to (\E_{/Z \times X}^S)^{core}$ will factor 
  through $map^S_\E(Z, Z \times X)$ and so it suffices to prove that 
  $$ \varphi: map_\E(Z,X) \to map^S_\E(Z, Z \times X)$$
  is an embedding.
  \par 
  We can actually do one more restriction. Let $map_\E(Z, Z \times X)_{id_Z}$ consist of the following path-components in 
  $map^S_\E(Z, Z \times X)$. A map $p: Z \to Z \times X$ is in $map_\E(Z, Z \times X)_{id_Z}$ if and only if $p \simeq (id_Z,g)$
  where $g: Z \to X$. Then the inclusion map from $map_\E(Z, Z \times X)_{id_Z}$ to $map^S_\E(Z, Z \times X)$ is an embedding.
  Moreover, the map $\varphi$ will factor again to a map $\varphi: map_\E(Z,X) \to map_\E(Z, Z \times X)_{id_Z}$.
  In order to finish the proof it suffices to show that this map is an equivalence.
  \par 
  Before we do that we will make one side note about $map_\E(Z, Z \times X)_{id_Z}$. Although it is a subspace of 
  $map^S_\E(Z, Z \times X)$, which is restricted by $S$, we did not include $S$ in the definition of  $map_\E(Z, Z \times X)_{id_Z}$.
  That is because {\it every} map in  $map_\E(Z, Z \times X)_{id_Z}$ is automatically in $S$ as it is the pullback of 
  $\Delta: X \to X \times X$ by the map $g \times id_X: Z \times X \to X \times X$.
  Thus $map_\E(Z, Z \times X)_{id_Z}$ is not restricted or related to $S$ in any way anymore.
  \par 
  Coming back to the proof, until now we constructed a map 
  $$\varphi: map_\E(Z,X) \to map_\E(Z,Z \times X)_{id_Z}$$
  and we want to prove that this map is an equivalence.
  Note that $map_\E(Z,Z \times X)_{id_Z}$ fits into the following homotopy pullback square.
  \begin{center}
   \pbsq{map_\E(Z, Z \times X)_{id_Z}}{map_\E(Z, Z \times X)}{\{ id_Z \} \times map_\E(Z,X)}{map_\E(Z,Z) \times map_\E(Z,X)}{}{\pi}{\simeq}{}
  \end{center}
  By the property of products the right hand map is an equivalence. This implies that we have an equivalence 
  $$\pi: map_\E(Z, Z \times X)_{id_Z} \xrightarrow{ \ \ \simeq \ \ } map_\E(Z,X)$$
  We can think of $\pi$ as the map that takes a map $(id_Z,g): Z \to Z \times X$ to the map $g: Z \to X$.
  Thus the map $\pi \varphi: map_\E(Z,X) \to map_\E(Z,X)$ is an equivalence. By $2$ out of $3$ we get that $\varphi$ is also an equivalence.
 \end{proof}

 \section{Application of the Yoneda Lemma}
 
 In the last section we look at couple applications of the Yoneda lemma in an elementary topos.
 \par 
 Let $X$ be a $0$-truncated object in an elementary higher topos $\E$. Then the map $X \to X \times X$ is an
 embedding in $\E$.Thus the map $X \times X \to \U$ stated in Theorem \ref{The Yoneda Lemma} factors through
 $X \times X \to \Omega$, where $\Omega$ is the subobject classifier in $\E$.
 This gives us a pullback square 
 \begin{center}
  \pbsq{X}{1}{X \times X}{\Omega}{}{}{}{}
 \end{center}
 Then Theorem \ref{The Yoneda Lemma} states that the map 
 $$X \to \Omega^X$$
 is an embedding. Thus this results recovers the case for elementary toposes as stated in Lemma \ref{Lemma Yoneda for ET}.
 \par 
 Let $X$ be a $-1$-truncated object in $\E$, meaning that the diagonal $\Delta: X \to X \times X$ is an equivalence.
 This means we have an equivalence
 $$map_\E(X , \U^X) \simeq map_\E( X \times X, \U ) \simeq map_\E (X , \U).$$
 The Yoneda map in $map_\E(X, \U^X)$ corresponds to the map $X \to \U$ classifying $id_X : X \to X$.
 Indeed we have the diagram of pullbacks
 \begin{center}
  \begin{tikzcd}[row sep=0.5in, column sep=0.5in]
   X \arrow[d, "id_X"] \arrow[r] \arrow[dr, phantom, "\ulcorner", very near start] 
   & X \arrow[d, "\Delta"] \arrow[r] \arrow[dr, phantom, "\ulcorner", very near start] 
   & \U_*\arrow[d] \\
   X \arrow[r] & X \times X \arrow[r] & \U 
  \end{tikzcd}
 \end{center}
 The Yoneda lemma now implies that the map $X \to \U$ classifying $id_X$ is an embedding.
 By \cite[Theorem 3.28]{Ra18c} this means $id_X$ is univalent.
 This gives us an alternative proof that the map is $id_X: X \to X$ is univalent if $X$ is $-1$-truncated.
 For the original argument see \cite[Example 6.25]{Ra18b}.


\begin{thebibliography}{9}
 
 
%
   
   
%
 
 
 
 
 
 
 
  

 
 
  
 
 
 

  \bibitem[MM92]{MM92}
  S. Mac Lane, I. Moerdijk. 
  {\it Sheaves in Geometry and Logic}, 
  Springer-Verlag, New York (1992).
 
 
 
 
  \bibitem[Ra18a]{Ra18a}
  N. Rasekh, 
  {\it An Introduction to Complete Segal Spaces}. 
  arXiv preprint arXiv:1805.03131 (2018)

  \bibitem[Ra18b]{Ra18b}
  N. Rasekh, 
  {\it Complete Segal Objects}. 
  arXiv preprint arXiv:1805.03561 (2018).
  
 \bibitem[Ra18c]{Ra18c}
 N. Rasekh, 
 {\it A Theory of Elementary Higher Toposes}. 
 arXiv preprint arXiv:1805.03805 (2018)
 
 \bibitem[Re01]{Re01}
 C. Rezk, 
 {\it A model for the homotopy theory of homotopy theory}, 
 Trans. Amer. Math.Soc., 353(2001), no. 3, 973-1007.
 
 

 
 

 


\end{thebibliography}
\end{document}